\documentclass[12pt,leqno]{amsart}
\usepackage{amssymb,amsthm,amsmath,latexsym}
\newtheorem{theorem}{\sc Theorem}[section]
\newtheorem{thm}[theorem]{\sc Theorem}
\newtheorem{lemma}[theorem]{\sc Lemma}
\newtheorem{prop}[theorem]{\sc Proposition}
\newtheorem{corollary}[theorem]{\sc Corollary}

 \title[Self-similar groups]{On a Class of Self-similar Polycyclic Groups}

\author{A. C. Dantas}
\address{Universidade de Bras\'ilia, Departamento de Matem\'atica, Bras\'ilia - DF,  70910-900, Brazil}
\email{alexcdan@gmail.com}

\author{E. de Melo} 
\address{Universidade de Bras\'ilia, Departamento de Matem\'atica, Bras\'ilia - DF, 70910-900, Brazil}
\email{emerson@mat.unb.br}

\author{R. N. de Oliveira}
\address{ Instituto de Matem\'atica e Estat\'istica, Universidade Federal de Goi\'as, Goi\^ania-GO, 74690-900 Brazil }
\email{ricardo@ufg.br}

\author{S. N. Sidki } 
\address{Universidade de Bras\'ilia, Departamento de Matem\'atica, Bras\'ilia - DF, 70910-900, Brazil}
\email{ssidki@gmail.com}


\begin{document}
\maketitle
\begin{center}
    \footnotesize{\bf{With an appendix by Bettina Eick}}
\end{center}
\maketitle
\begin{abstract}
A group $G$ is self-similar if it admits a triple $(G,H,f)$ where $H$ is a subgroup of $G$ and $f: H \to G$ a simple homomorphism, that is, the only subgroup $K$ of $H$, normal in $G$ and $f$-invariant ($K^f \leq K$) is trivial. The group $G$ then has two chains of subgroups:
\[
G_0 = G,\ H_0 = H,\ G_k = (H_{k-1})^f,\ H_k = H \cap G_{k}\ \text{for } (k \geq 1).
\]

We define a family of self-similar polycyclic groups, denoted $SSP$, where each subgroup $G_k$ is self-similar with respect to the triple $(G_k , H_k, f)$ for all $k$. By definition, a group $G$ belongs to this $SSP$ family provided $f: H \rightarrow G$ is a monomorphism, $H_k$ and $G_{k+1}$ are normal subgroups of index $p$ in $G_k$ ($p$ a prime or infinite) and  $G_k=H_kG_{k+1}$. When $G$ is a finite $p$-group in the class $SSP$, we show that the above conditions follow simply from $[G:H] = p$ and $f$ is a simple monomorphism.

We show that if the Hirsch length of $G$ is $n$, then $G$ has a polycyclic generating set $\{a_1, \ldots, a_n\}$ which is self-similar under the action of $f: a_1 \rightarrow a_2 \rightarrow \ldots \rightarrow a_n$, and then $G$ is either a finite $p$-group or is torsion-free. Surprisingly, the arithmetic of $n$ modulo $3$ has a strong impact on the structure of $G$. This fact allows us to prove that $G$ is nilpotent metabelian whose center is free $p$-abelian ($p$ prime or infinite) of rank at least $n/3$.

We classify those groups $G$ where $H$ has nilpotency class at most $2$. Furthermore, when $p=2$, we prove that $G$ is a finite $2$-group of nilpotency class at most $2$, and classify all such groups.

\end{abstract}
\maketitle
\section{Introduction}

It has been customary to define self-similarity for a group $G$ as the existence of both a subgroup $H$ of $G$ having finite index $m$ in $G$, and of a homomorphism $f: H \to G$. Then the group $G$ is self-similar provided the only normal subgroup $K$ of $G$ which is $f$-invariant subgroup (that is, $K^f \leq K$) is trivial. In this case, the triple $(G,H,f)$ is called simple. Such a simple triple leads to a faithful representation of $G$ as a state-closed group of automorphisms of a regular one-rooted tree of degree $m$, having as basis the set of right cosets of $H$ in $G$. The restriction on the finiteness of $m$ is unnecessary. The Kaloujnine-Krasner wreath embedding theorem $G \to H \wr (G/H)$ 
extends recursively, using $f$, to a representation  of $G$ into the automorphism group of the one rooted regular tree of valency $m$. The simplicity of self-similarity ensures that the representation is faithful (see \cite{KK}). Strictly speaking, this is the notion of transitive self-similar group, where the action of the group on the first level of the tree is transitive. A generalization to intransitive actions was introduced recently in \cite{DSS}. Self-similarity for infinite index $m$ was pursued sporadically in \cite{S1} from 1984,  \cite{S2} from 2004, and  \cite{S3} from 2005.

The following result, which holds for a general $m$ is the starting point for the developments in this paper.

\begin{prop}[\cite{BeS}, Proposition 23]\label{proBeS}
Let $G$ be a group, $H$ a subgroup of $G$, and $(G,H,f)$ a simple triple. Denote $H^f$ by $G(1)$, $H \cap H^f$ by $H(1)$, the center of $G$ by $Z(G)$ and $f|_{H(1)}$ by $f_1$. Suppose that $G$ decomposes as $G = Z(G) H H^f$. Then, $(G(1), H(1), f_1)$ is a simple triple.    
\end{prop}

The result fails to be inductive, as we cannot guarantee a corresponding factorization of  $G(1)$. Our definition of the class $SSP$ of polycyclic groups, imposes such a factorization. Specifically, given a self-similar group $G$ with a simple triple $(G,H,f)$ and the chains of subgroups:
\[
G_0 = G,\ H_0 = H,\ G_k = (H_{k-1})^f,\ H_k = H \cap G_{k}\ \text{for } (k \geq 1).
\] we define $G$ as an $SSP$ group if:

\begin{enumerate}
    \item $f: H \to G$ a monomorphism;
    \item $(G_k, H_k, f)$ is a simple triple for any $k$;
    \item $H_k$ and $G_{k+1}$ are normal subgroups of $G_k$, and $G_k = H_kG_{k+1}$;
    \item $G_k/H_k$ is cyclic of order $p$, a prime number or infinite.
\end{enumerate}

The group $G$ will be said to be an $SSP$-lifying, or simply lifting, of $H$. We have directly from the definition that a finite $SSP$ group is a finite $p$-group admitting an injective virtual endomorphism of degree $p$. The converse also holds.

\begin{theorem}\label{p-group}
Let $G$ be a finite $p$-group with a simple triple $(G,H,f)$ where $f$ is injective and $[G:H]=p$. Then $G$ is an $SSP$ group.
\end{theorem}

By the Auslander-Swan theorem, polycyclic groups of finite Hirsch length have faithful representations in $\mathrm{GL}(n, \mathbb{Z})$ for some $n$ (see \cite[Theorem 3.3.1]{LR}). It was shown in \cite{BS} that the affine group $\mathrm{Affine}(n, \mathbb{Z}) = \mathbb{Z}^n \rtimes \mathrm{GL}(n , \mathbb{Z})$ has a faithful representation on an $m$-ary tree (where $m$ is finite) which is both finite-state and self-similar. Being of finite-state is inherited by subgroups, but this is not true for self-similarity in general; though, for $\mathrm{Affine}(n, \mathbb{Z})$, we observe from its action on the tree that its restriction to any subgroup containing the module $\mathbb{Z}^n$ induces faithful self-similarity of the subgroup.

It is known from \cite{BeS} that finitely generated torsion free nilpotent groups of class at most 2 afford faithful transitive self-similar representations on $m$-trees, where $m$ is finite. Furthermore, Berlatto and dos Santos have presented in \cite{BeS2} a detailed proof that the $N_{3,4}$ group of Bludov-Gusev, which is 4-generated torsion-free nilpotent of class 3, is not self-similar, be it transitive or not, for any finite $m$.

There are very few general results where a self-similar group of a certain variety embeds constructively in a larger self-similar group of the same variety. One such general result is

\begin{prop}[\cite{LS}, Section 3]\label{prop3}
Suppose that $(G, H, f)$ is a simple triple where the index $[G:H]= m$ is finite or infinite. Consider $\widehat{H} =\underset{k-1}{ G   \times \cdots \times G}\times H,$ and define the homomorphisms $F: \widehat{H} \to G^k$, by
$F: (g_1, \ldots, g_{k-1},h) \mapsto ( h^f,g_1, \ldots, g_{k-1}).$ Then, the triple $(G^k,\widehat{H},F)$ is simple.
\end{prop}

In the above proposition, if $f$ is a monomorphism, then $F$ is also a monomorphism. Consequently, this method provides a way to construct new SSP groups.

This paper is organized as follows: Section 2 provides a proof of Theorem \ref{p-group} and an exposition of preliminary properties of $SSP$-groups of Hirsch length $n$. Notable among these properties are the following:

\begin{theorem}
Let $G$ be an $SSP$-group of Hirsch length $n$.
\begin{enumerate}
\item[(i)] There exists a polycyclic generating subset $S = \{a_1,\dots,a_n\}$ of $G$ such that its elements have the same order $p$ (finite or infinite), and $S$ is self-similar under $f: a_i  \to  a_{i+1}$ for $1\leq i\leq n-1$; 
\item[(ii)] $G$ is a nilpotent group.
\end{enumerate}  
\end{theorem}

Considering the generator set $S$ from the preceding theorem, let $c$ be the maximum index $s$ such that $[a_1, a_2] = \dots= [a_1, a_s] =e$. Then $c$ will be called the {\it cut-off} point for the group $G$, and $G$ is then said to be of $G(n, c)$-type. Given a non-abelian group $G$ of $G(n, c)$-type, its subgroups $H$ and $H^f$ have the same cut-off point $c$, and both are $SSP$-groups of type $G(n-1,c)$.

In Section 3, we construct groups of $G(n, c)$-type  for $2\leq n \leq6$. In particular, we prove that $G(4, 2)$-type groups do not exist. Furthermore, we show that for $n \leq5$ the groups have nilpotency class at most 2, and there exist groups of $G(6, 4)$-type of nilpotency class 3. We comment that Hand-calculations of groups of $G(n, c)$-
type were done up to $n = 12$, and no groups with a nilpotency class greater than 3 were found. This leads to the problem about the existence of an upper bound for the nilpotency class of such groups in general and whether, indeed, the upper bound is 3.

In Section 4, we analyze the structure of groups of $G(n, c)$-type under the additional condition of being metabelian, and then progress to prove that all $SSP$-groups are indeed metabelian, having a triple factorization determined by the modulo 3 arithmetic of $n$.

\begin{theorem}
Let $G$ be a $G(n, c)$-group, $n\geq3$, and write $n = 3q(n) +r(n)$, where $q(n)$, $r(n)$ are non-negative integers and $r(n) \leq 2$.
Then,
\begin{enumerate}
    \item[(i)] $G$ is metabelian;
    \item[(ii)] Let  $n\geq3$. Then $c =c(n,j)= 2q(n) +r(n) +j$,  for some $0 \leq j \leq q(n)$;
    \item[(iii)] Let $G$ be a non-abelian.  Then any sequence of $SSP$-liftings of $G$ terminates after a finite number of steps. 
    \item[(iv)] $[a_1, a_{c(n,j)+1}]$ is a non-trivial element of $\langle a_{q(n)+1-j},\dots,a_{q(n)+1+r(n)+2j}\rangle$;    
    \item[(v)] $Z(G) = \langle a_{q(n)+1 -j} ,\dots, a_{c(n, j)}\rangle$ has rank at least  $ n/3$;
    \item[(vi)] $G$  has the triple decomposition $$G = \langle a_1,\dots,a_{q(n)-j}\rangle \cdot Z(G) \cdot \langle a_{c(n,j)+1},\dots,a_n\rangle,$$ where the three factors are free abelian $p$-groups.
\end{enumerate}
\end{theorem}

In Section 5  we provide faithful matrix  presentations for class 2 nilpotent groups of $G(n,c)$-type. 

In Section 6 we prove

\begin{theorem}
Let $G$ be a group of $G(n+1, c)$-type, and suppose its subgroup $H$ has nilpotency class 2. Then, $G$ is nilpotent of class at most 3.    
\end{theorem}

As a consequence, we obtain

\begin{corollary}
Let $G$ be an SSP-group where $p = 2$. Then $G$ is a finite 2-group of nilpotency class at most 2.    
\end{corollary}

\section{Preliminary Results on $SSP$-groups}

In this section, we prove Theorem \ref{p-group} and list some results on $SSP$-groups that will be used throughout the text without explicit reference.

The following lemma is a straightforward consequence of Proposition \ref{proBeS}; we include a proof for convenience.

\begin{lemma}\label{virt2}
Let $G$ be a finite $p$-group with a simple triple $(G,H,f)$ where $[G:H]=p$. Then $f:H\cap H^{f} \rightarrow H^{f}$ is also a simple virtual endomorphism of degree $p$.
\end{lemma}
\begin{proof}
Since $[G:H]=p$ we have that $H \vartriangleleft G$ and $G = HH^{f}$. Moreover, since $|\frac{H^{f}}{H \cap H^{f}}|=|\frac{HH^{f}}{H}|$ we have that $H\cap H^f$ has index $p$ in $H^f$.  Let $K$ be a subgroup of $H\cap H^{f}$, normal in $H^{f}$, with $K^{f}\leq K$. Note that $K^{H}$ is a normal subgroup of $G$. Let $k^h\in K^H$ where $k\in K$ and $h\in H$. Applying $f$ on $k^{h}$ we obtain that $(k^h)^f=(k^f)^{h^f}\in K^{H^f}=K$. Then $(K^H)^f\leq K $. Since $f$ is a simple endomorphism we conclude that $K=1$, as desired.
\end{proof}

{\bf Proof of Theorem \ref{p-group}}:
\begin{proof}
Since $f$ is injective and $[G:H]=p$ we have that $H$ and $H^f$ are both maximal subgroups of $G$ and then normal in $G$ with $G=HH^f$. Now, as $|H^f|< |G|$ we use induction on the order of $G$ and apply Lemma \ref{virt2} to deduce that $H^f$, and consequently $G$, are $SSP$ groups.     
\end{proof}

\begin{thm}
    Let $G$ be an $SSP$-group of Hirsch length $n$. Then, the following results hold.
    \begin{enumerate}
        \item The trivial group is the only $f$-invariant subgroup of $G$.
        \begin{proof}
        This is because $f$ is injective and $G_{i+1}=(H_i)^f$ is not contained in $H$ for any $1\leq i\leq n-1$.
        \end{proof}
        \item There exists a generating subset $S = \{a_1,\ldots, a_n\}$ of $G$ such that $a_{i+1}=(a_i)^f$ for $1\leq i \leq n-1$. The set $S$ is called the canonical basis for $G$. Furthermore, $S_i = \{a_{i+1},..., a_{n}\}$ is a generating set for $G_i$ and $\{a_{i+1},..., a_{n-1}\}$ is a generating set for $H_i$. Strictly speaking, it is the set of cyclic subgroups $\{\langle a_1\rangle,\ldots, \langle a_n\rangle \} $ that is canonical.
        \begin{proof}
        Note that since $f$ is injective $(G,H^f,f^{-1})$ is a simple triple. Using induction on $n$ we have that $G_1=H^f$ has a basis $\{a_2,\ldots,a_n \}$ with the desired properties. Now, we define $a_1=(a_2)^{f^{-1}}$.   
        \end{proof}
        \item For $n\geq 3$, $[a_i, a_{i+1}]$  = 1 for all $i$. If $[a_1, a_2] =\cdots=[a_1, a_s] =1$, then $\{a_1,\ldots, a_s\}$ is a commutative set.
        \begin{proof}
        $G_{n-2}=\langle a_{n-1},a_{n}\rangle$ is abelian and then $[a_{n-1},a_{n}]=1$. Now, since $[a_i, a_{i+1}]^{f^{n-(i+1)}}=[a_{n-1},a_{n}]$ the result follows.
        \end{proof}
       \item Given a word $W(a_{i_1},\ldots,a_{i_s})$ in $G$ then $W(a_{1},\ldots,a_{l})^f=W(a_{2},\ldots,a_{l+1})$, which is defined provided $l\leq n-1$.  Whenever necessary, we denote $[a_1,a_i]^{f^t}$ by $[a_1,a_i]_ t$.
    \begin{proof}
    This is because $f$ is injective.
    \end{proof}
     \item Groups of SSP-type are nilpotent.
        \begin{proof}
        Let $G$ be of $G(n, c)$-type. Proceed by induction on $n$. The assertion is clearly true for $n=2$. Since $H$, $H^f$ are normal subgroups of $G$, and of $G(n-1, c)$ type, by induction, they are both nilpotent. The result follows since $G = H  H^f$.
    \end{proof}
    \end{enumerate}
\end{thm}

Define $c$ as the maximum index $s$ such that $[a_1, a_2] = \cdots = [a_1, a_s] = 1$. Then, $2 \leq c \leq n$. We call $c$ the {\it cut-off} point of the group $G$, and say $G$ is of $G(n,c)$-type. We note that $G$ abelian is equivalent to $c=n$. The cut-off point $c$ is an important parameter in an algorithmic construction of non-abelian $SSP$-groups. Specifically, for a non-abelian $SSP$-group $G$ of type $G(n,c)$, its subgroup $H$ inherits the same cut-off point $c$. Consequently, $H$ is an $SSP$-group of type $G(n-1,c)$.

\begin{theorem}
 For $n \geq 3$, the polycyclic collection formulas of $G$ are consequences of the closure under the application of powers of $f$ to the set of relations 
\begin{itemize}
    \item[i)] $a_1^p=e$ (where $p$ a prime number or infinite)
    \item[ii)] $[a_1, a_j] = e \quad (2 \leq j \leq c),$
    \item[iii)] $[a_1, a_j] = a_2^{i_2} \dots a_{j-1}^{i_{j-1}} \quad (c+1 \leq j \leq n)$ 
\end{itemize}
called the essential relations of $G$ with respect to the set of generators $S=\{a_1,\ldots, a_n\}$.
\end{theorem}
 \begin{proof}
For a general polycyclic group $P$ with subnormal series $\{P_i\ (0 \leq i \leq n-1)\}$ and polycyclic basis $B = \{b_1, b_2, \dots, b_n\}$, where $b_i$ generates $P_i$ modulo $P_{i+1}$ ($0 \leq i \leq n-1$), its presentation with respect to $B$ is expressed by polycyclic collection formulas having one of three forms:
    \begin{enumerate}
        \item $b_i^{m_i} = b_{i+1}^{u_{i+1}} \dots b_n^{u_{n}}$ (if $m_i$ is finite),
        \item $b_{i+1} b_i = b_i^{v_{i}} b_{i+1}^{v_{i+1}} \dots b_n^{v_{n}}$,
        \item $b_{i+1}^{-1} b_i = b_i^{w_{i}} b_{i+1}^{w_{i+1}} \dots b_n^{w_{n}}$,
    \end{enumerate}
    where $u_j$, $v_j$, $w_j$ are integers.

    In the case of $SSP$-groups, since $p = o(a_i)$ is either a prime number or infinite, the first formula is superfluous. The second formula follows from
    \[
    [a_1, a_n] \in \langle a_2, \dots, a_{n-1} \rangle
    \]
    and by induction on $n$. While the third formula follows from
    \[
    [(a_1)^{-1}, a_n] \in [a_n, a_1] \langle a_2, \dots, a_{n-1} \rangle = \langle a_2, \dots, a_{n-1} \rangle
    \]
    and by induction on $n$.
    \end{proof}

\section{First Examples}\label{examples}
\begin{enumerate}
    \item Let $G$ be an $SSP$-group, $n \geq 3$, and the cut-off point $c = n-1$. Then, the essential relations are
\begin{itemize}
    \item[i)] $a_1^p=e,$
    \item[ii)] $[a_1, a_i] = e \quad (2 \leq i \leq n-1),$
    \item[iii)] $[a_1, a_n] = a_2^{k_2} \dots a_{n-1}^{k_{j-1}}\neq e$ .
\end{itemize}
    for some integers $k_2, \dotsc, k_{n-1}$.

Clearly, $G$ splits as $\langle a_1, \dotsc, a_{n-1} \rangle \rtimes \langle a_n \rangle$. Furthermore, $G$ has nilpotency class $2$ and
    \[
    Z(G) = \langle a_2, \dotsc, a_{n-1} \rangle.
    \]

    \item There are no groups of $G(4, 2)$-type. Thus, for a non-abelian $SSP$-group $G$ of rank $4$, 
    $H = \langle a_1, a_2, a_3 \rangle$ is abelian and the cut-off point  $c$ of $G$  is  $c = 3.$
    
\begin{proof}
     Let $G$ be an $SSP$-group of $G(4, 2)$-type. We have $[a_1, a_{3}]=(a_2)^x\neq e$ and so $[a_2, a_{4}]=(a_3)^x\neq e$. Furthermore, we have $[a_1, a_{4}]=(a_2)^u(a_3)^v$. The Hall-Witt identity $$[a_1, a_3,a_4]\cdot [a_4,a_1,a_3]\cdot [a_3,a_4,a_1]=e$$ holds. The second and third terms are trivial, thus $[a_1, a_3,a_4]=[(a_2)^x,a_4]=[a_2,a_4]^x=(a_3)^{x^2}=e$, that is a contradiction.
\end{proof}

    \item The possible cut-off points for $n = 5$ are $c=4$ and $5$,  and we know that the corresponding groups are realizable.
    \begin{proof}
   Let $G$ be an $SSP$-group of $G(5,3)$-type. Then, $$[a_1 , a_3] = e,$$ $$[a_1 , a_4] = a_2^u a_3^v\neq e$$ $$[a_2 , a_5] = a_3^u a_4^v\neq e$$ $$[a_1 , a_5] = a_2^x a_3^y a_4^z.$$ Then, $G' \leq \langle a_2, a_3, a_4\rangle,$ is an abelian group. The Hall-Witt relation $$[a_1 , a_4, a_5]\cdot[a_5 , a_1, a_4]\cdot[a_4 , a_5, a_1] = e$$ holds. As the second and third terms are trivial, we have $$e = [a_1 , a_4, a_5] = [a_2^u a_3^v, a_5]= = [a_2^u, a_5] = (a_3^u a_4^v)^u = a_3^{u^2} a_4^{uv} $$ $u=0$.

Similarly, the Hall-Witt relation $$[a_1 , a_2, a_5]\cdot [a_5 , a_1, a_2]\cdot [a_2 , a_5, a_1] = e,$$ implies $v=0$. Thus, $[a_1 , a_4] = e$; that is a contradiction.     
    \end{proof}
    \item Let $n = 6$. Suppose $G$ is of $G(6, c)$-type. Then $G_1$ is of $G(5, 4)$ or $G(5, 5)$-type, and therefore for $G$, the possible cut-off points are $c = 4, 5, 6$. We know that $G$ is realizable for $c = 5, 6$.

    \begin{enumerate}
        \item Suppose $G$ is an SSP-group of $G(6, 4)$-type. Then,
        \[
        [a_1, a_i] = e \quad \text{for } 1 \leq i \leq 4;
        \]
    there exist integers $u_2, u_3, u_4$ such that
        \[
        [a_1, a_5] = a_2^{u_2} a_3^{u_3} a_4^{u_4} \neq e
        \]
        \[
        [a_2, a_6] = a_3^{u_2} a_4^{u_3} a_5^{u_4} \neq e
        \]
        \[
        [a_1, a_6] = a_2^{x_2} a_3^{x_3} a_4^{x_4} a_5^{x_5}.
        \]
        We note that
        \[
        H_1 = H \cap H^f = \langle a_2, a_3, a_4, a_5 \rangle
        \]
        is an abelian normal subgroup of $G$, and $G / H_1$ abelian; that is, $G$ is metabelian. We note further that $\langle a_3, a_4 \rangle$ is central in $G$.
        Now, as we have done in the analysis of groups of $G(5,3)$-type, we employ the two Hall-Witt relations
        \[
         [a_1, a_5, a_6][a_6, a_1, a_5][a_5, a_6, a_1] = e,
        \]
        \[
         [a_1, a_2, a_6][a_6, a_2, a_1][a_1, a_6, a_2] = e,
        \]
        to conclude that
        \[
        u_2 = 0 = u_4.
        \]
        Therefore,
        \[
        [a_1, a_5] = a_3^{u_3}, \quad u_3 \ne 0,
        \]
        \[
        [a_2, a_6] = a_4^{u_3},
        \]
        and
        \[
        [a_1, a_6] = a_2^{x_2} a_3^{x_3} a_4^{x_4} a_5^{x_5}
        \]
        for some integers $x_i$, $(2 \leq i \leq 5)$. We compute,
        \[
        [a_1, a_6, a_1] = [(a_5)^{x_5}, a_1] = [a_5, a_1]^{x_5} = a_3^{-u_3 x_5},
        \]
        \[
        [a_1, a_6, a_6] = [(a_2)^{x_2}, a_6] = ([a_1, a_5]_1)^{x_2} = a_4^{u_3 x_2}.
        \]
        Therefore, $G$ has nilpotency class at most $3$, where
        \[
        \gamma_2(G) = \langle a_3^{u_3}, a_4^{u_3}, a_2^{x_2}, a_5^{x_5} \rangle,
        \]
        \[
        \gamma_3(G) = \langle a_3^{-u_3 x_5}, a_4^{u_3 x_2} \rangle.
        \]
        In particular, $\gamma_3(G)$ is trivial if and only if $x_2 = x_5 = 0$.

        \item Suppose $o(a_i) = 2$. Then as $\{a_2, a_3, a_4, a_5\}$ is commutative,
        \[
        [a_1, a_6] = a_2^{x_2} a_3^{x_3} a_4^{x_4} a_5^{x_5}
        \]
        is either trivial or an involution, and therefore $\langle a_1, a_6 \rangle$ is either abelian or dihedral of order $8$. Thus, $\gamma_3(G)$ is trivial.

        \item The group $G$ decomposes as
        \[
        G = (\langle a_1, a_2, a_3, a_4 \rangle \rtimes \langle a_5 \rangle) \rtimes \langle a_6 \rangle.
        \]
        where
        \[
        [a_i, a_5] = e \quad (2 \leq i \leq 4),
        \]
        \[
        a_1^{a_5} = a_1 a_3^{u_3}, \quad u_3 \ne 0,
        \]
        \[
        [a_i, a_6] = e \quad (3 \leq i \leq 5),
        \]
        \[
        a_1^{a_6} = a_1 a_2^{x_2} a_3^{x_3} a_4^{x_4} a_5^{x_5},
        \]
        \[
        a_2^{a_6} = a_2 a_4^{u_3}.
        \]

        \item We construct $G$ starting anew with the presentation
        \[
        H = \langle a_1, a_2, a_3, a_4, a_5 \mid a_i^p = e, [a_i, a_5] = e \quad (2 \leq i \leq 4), [a_1, a_5] = a_3^{u_3}, \quad u_3 \ne 0 \rangle,
        \] and by defining the action of $a_6$ on $H$ as an extension of
        \[
        [a_i, a_6] = e \quad (3 \leq i \leq 5),
        \]
        \[
        a_1^{a_6} = a_1 a_2^{x_2} a_3^{x_3} a_4^{x_4} a_5^{x_5},
        \]
        \[
        a_2^{a_6} = a_2 a_4^{u_3}.
        \]
        To prove that $a_6$ is an automorphism of $H$ we check
        \begin{enumerate}
            \item[(i)] $\{a_1, a_2, a_3, a_4, a_5\}^{a_6}$ generates $H$,
            \item[(ii)] $[a_i, a_5]^{a_6} = e \quad (2 \leq i \leq 4)$,
            \[
            [a_1, a_5]^{a_6} = (a_3^{u_3})^{a_6}.
            \]
        \end{enumerate}
        Part (i) is straightforward. Part (ii) reduces to checking
        \[
        [a_2, a_5]^{a_6} = e, \quad [a_1, a_5]^{a_6} = a_3^{u_3}
        \]
        which are also clear. To prove $o(a_6) = p$, we note that the centralizer of $a_2$ in $\langle a_6 \rangle$ is trivial, since
        \[
        (a_2)^{a_6^k} = a_2 a_4^{k u_3}
        \]
        for all integers $k$, and from $u_3 \ne 0$. Finally, we show that the function $f$ defined on $\{a_1, \dotsc, a_6\}$ by $a_i \mapsto a_{i+1}$ $(1 \leq i \leq 5)$ extends to the intended similarity on $G$.

        \item Let $p$ be odd or infinite. Then, $G$ has nilpotency class $3$ if and only if the pair $(x_2, x_5)$ is different from $(0, 0)$.
    \end{enumerate}
\end{enumerate}

\section{Metabelian $SSP$-groups}
The Hall-Witt identity for a metabelian group $G$ has the simplified form
\[
[x, y, z][z, x, y][y, z, x] = e
\]
for all $x, y, z$ in $G$. From the Hall-Witt identity, we will be using in our analysis of metabelian groups of $G(n+1, c)$-type the relations
\[
W(1, i, n+1):\quad [a_1, a_i, a_{n+1}] [a_{n+1}, a_1, a_i]  [a_i, a_{n+1}, a_1] = e,
\]
where $1 \leq i \leq n+1$. Given the symmetry $i \mapsto n+2 - i$ on the interval $[2, n]$, we refer to the pair $W(1, i, n+1)$, $W(1, n+2 - i, n+1)$ as {\it symmetric}.

The relations $W(1, i, n+1)$ will be important in identifying $c$ as a numerical function $c(n, j)$ in terms of the decomposition
\[
n = 3 q(n) + r(n),\quad \text{where } q(n) \geq 0,\ r(n) = 0,1,2.
\]
Define the numerical function
\[
c(n, 0) = 2 q(n) + r(n)
\]
\[
c(n, j) = c(n, 0) + j,\quad \text{where } 0 \leq j \leq q(n),
\]
Note that $c(n, q(n)) = n$.

Given a metabelian $G(n, c)$-type group $G$. It will be shown that $c = c(n, j)$ for some $0 \leq j \leq q(n)$.

Computations in Section 3 confirm the coincidence of $c$ with the corresponding values of $c(n, j)$:
\[
c(3, 0) = 2 = c(2, 0);\quad c(4, 0) = 3 = c(3, 1);\quad c(5, 0) = 4 = c(4, 1) = 4,
\]
\[
c(6, 0) = 4,\quad c(6, 1) = 5,\quad c(6, 2) = 6.
\]

\begin{theorem}\label{cutpoint}
Let $G$ be a metabelian group of type $G(n, c)$, $n \geq 3$. Then,
\begin{enumerate}
    \item[(1)] $c = c(n, j)$ for some $0 \leq j \leq q(n)$,
    \item[(2)] $[a_1, a_{c+1}] \in \langle a_{q(n)+1 - j}, \dotsc, a_{q(n)+1 + r(n) + 2j } \rangle$, a non-trivial element of $G$.
    \item[(3)] Let $G$ be non-abelian. Then sequences of liftings of $G$ have a bounded number of steps.
\end{enumerate}    
\end{theorem}
\begin{proof}
The assertion is true for $3 \leq n \leq 6$, and for $G$ of type $G(n,c)$, where $c = n-1$ (that is, $j = q(n) - 1$). (see the examples).

Assume assertions (1) and (2) to be true for $n$. We may assume $0 \leq j \leq q(n) -1$.

Suppose $G$ lifts to $\widehat{G}$ an $SSP$-group of type $G(n+1, c)$. Then, we have in $\widehat{G}$,

\begin{enumerate}
    \item[(i)] $[a_1, a_{c+1}]= a_{q(n)+1 - j}^u\ a_{q(n)+2 - j}^{u'} \dotsm\ a_{q(n)+1 + r(n) + 2j - 1}^{v'}\ a_{q(n)+1 + r(n) + 2j}^{v}\neq e$,

    \item[(ii)] $[a_1, a_{n+1}] = a_2^{z_2} \dotsm a_n^{z_n}$.
\end{enumerate}

Consider
\[
W(1, c+1, n+1):\quad [a_1, a_{c+1}, a_{n+1}] [a_{n+1}, a_1, a_{c+1}] [a_{c+1}, a_{n+1}, a_1] = e.
\]
Evaluate the first of the three terms:
\begin{align*}
[a_1, a_{c+1}, a_{n+1}] &= [ a_{q(n)+1 - j}^u\ a_{q(n)+2 - j}^{u'} \dotsm\ a_{q(n)+1 + r(n) + 2j - 1}^{v'}\ a_{q(n)+1 + r(n) + 2j}^{v},\ a_{n+1} ] \\
&= [ a_{q(n)+1 - j}^u,\ a_{n+1} ] \\
&\quad \text{(since $a_{n+1}$ commutes with the elements $a_{q(n)+2 - j},\ \dotsc,\ a_{q(n)+1 + r(n) + 2j}$)} \\
&= [ a_1^u,\ a_{n+1 - q(n) + j} ]_{q(n) - j} \\
&= [ (a_1)^u,\ a_{2 q(n) + r(n) + j +1} ]_{q(n) - j} \\
&= [ a_1,\ a_{c +1} ]^{u}_{q(n) - j}.
\end{align*}
The next term:
\[
[ a_{n+1},\ a_1,\ a_{c+1} ] = [ a_2^{z_2} \dotsm a_n^{z_n},\ a_{c+1} ] = [ a_1^{z_2} \dotsm a_{n-1}^{z_n},\ a_c ]_{1} = e.
\]
Now, since
\[
[ a_{c+1},\ a_{n+1} ] = [ a_1,\ a_{n - c +1} ]_c = [ a_1,\ a_{n - (2 q(n) + r(n) + j) +1} ]_c
\]
\[
= [ a_1,\ a_{q(n) +1 - j} ]_c,
\]
and $q(n) +1 - j \leq c$, we have $[ a_{c+1},\ a_{n+1} ] = e$, and the third term is
\[
[ a_{c+1},\ a_{n+1},\ a_1 ] = e.
\]
Therefore, the above Hall-Witt relation reduces to
\[
[ a_1,\ a_{c+1},\ a_{n+1} ] = [ a_1,\ a_{c +1} ]^{u}_{q(n) - j} = e,
\]
\[
[ a_1,\ a_{c +1} ]^{u} = e,
\]
and thus $u = 0$, follows, since $[ a_1,\ a_{c +1} ]$ is non-trivial.

Likewise, expanding the relation $W(1, n+1 - c, n+1)$ symmetric to $W(1, c+1, n+1)$, yields
\[
v = 0.
\]
Hence, in $\widehat{G}$, the expression in (i) reduces to
\[
(ii)^{*} \quad [a_1, a_{c+1}] = a_{q(n)+2 - j}^{u'} \dotsm a_{q(n)+1 + r(n) + 2j - 1}^{v'}\quad \text{non-trivial}.
\]
\begin{lemma}
There exists $0 \leq j' \leq q(n+1) -1$ such that
\[
q(n)+2 - j = q(n+1) +1 - j',\] and, \[ q(n)+1 + r(n) + 2j -1 = q(n+1) +1 + r(n+1) + 2j'. \]   
\end{lemma}

In other words,
$ j' = q(n+1) - q(n) + (j -1)$  implies that  $q(n)+1 + r(n) + 2j -1 = q(n+1) +1 + r(n+1) + 2j'.$

\begin{proof}
Using that $3 q(n+1) + r(n+1) = n+1$, we have that
\begin{align*}
q(n+1) +1 + r(n+1) + 2j' &= q(n+1) +1 + r(n+1) + 2 ( q(n+1) - q(n) + (j -1) )  \\
&= q(n+1) +1 + r(n+1) + 2 q(n+1) - 2 q(n) + 2 (j -1) \\
&= 3 q(n+1) + r(n+1) - 2 q(n) + 2 (j -1) +1 \\
&= q(n) +1 + r(n) + 2j -1.
\end{align*}
\end{proof} 
If $j' < 0$, then $j = 0$ and $q(n+1) = q(n)$. The last is equivalent to $n = 3 q(n)$ or $3 q(n) +1$. Thus, from (i), $[ a_1, a_{c+1} ]$ is non-trivial and it belongs to the subgroups $\langle a_{q(n)+1} \rangle$, or $\langle a_{q(n)+1}, a_{q(n)+2} \rangle$. However, by $(ii)^{*}$, both possibilities collapse in $\widehat{G}$, which is a contradiction.

Since $c' = c = 2 q(n) + r(n) + j$ and $j' = q(n+1) - q(n) + (j -1)$ we reach $c' = 2 q(n+1) + r(n+1) + j'$. In view of the above lemma, we have in $\widehat{G}$,
\[
[ a_1, a_{c+1} ] = [ a_1, a_{c'+1} ] = a_{q(n+1) +1 - j'}^x \dotsm a_{q(n+1) +1 + r(n+1) + 2j'}^y \quad \text{non-trivial}.
\] This proves item $(2)$. As for item $(3)$, this simply follows from the fact that given $c$, the number of solutions $(n,j)$ for $c(n,j)=c$ is finite.
\end{proof}

\subsection{The center $Z(G)$ of $G$}

\begin{prop}
 Let $G$ be a metabelian group of $G(n, c)$-type. Define

$\Sigma_1(n, j) = \langle a_1, \dots, a_c \rangle$ and  $\Sigma_n(n, j) = \Sigma_1(n, j)^{f^{q(n)-j}} = \langle a_{q(n)+1-j}, \dots, a_n \rangle$.

Then,
\begin{enumerate}
    \item $\Sigma_1(n, j)$ and $\Sigma_n(n, j)$ are free abelian $p$-groups of rank $c$ and are  maximal abelian in $G$;
    \item $Z(G) = \Sigma_1(n, j) \cap \Sigma_n(n, j)
    = \langle a_{q(n)+1-j}, \dots, a_c \rangle$  has rank at least  $n/3$;
    \item $G = \Sigma_1(n, j) \cdot \Sigma_n(n, j)
    = \langle a_1, \dots, a_{q(n)-j} \rangle \, Z(G) \, \langle a_{c+1}, \dots, a_n \rangle$;
    \item  If $G$ is non-abelian, then
$[a_1, a_{c+1}]_{k}$  is a linearly independent subset  of $Z(G)$ for all $0 \leq k \leq q(n) - j-1$.
\end{enumerate}   
\end{prop}

\begin{proof}    
(1) The first part is clear. We now prove that $\Sigma_1(n, j)$ and $\Sigma_n(n, j)$ are  maximal abelian in $G$. The proof proceeds by induction on $n$. The assertion is true for $n \leq 2$. Let $n \geq 3$ and $G$ be non-abelian. Let 
\[
x = a_{c(n,j)+1}^{m_1} \dots a_{n}^{m_n}.
\]
Suppose $\langle a_1, \dots, a_{c(n,j)}, x \rangle$ is abelian. Note that $\langle a_{c(n,j)+1}, \ldots, a_n\rangle$ is an abelian subgroup, and therefore $x$ commutes with $a_{c(n,j)+1}$. On the other hand, $\langle a_2, \dots, a_{c(n,j)+1}\rangle$ is a maximal abelian subgroup of $\langle a_2, \dots, a_{n} \rangle$. By induction, $\langle a_2, \dots, a_{c(n,j)+1}, x \rangle$ is not abelian, which leads to a contradiction. 

(2) We know that $\Sigma_1(n, j)$ is abelian and that $c = c(n, j) = 2q(n) + r(n) + j$. Then,
\[
Z(n, j) = \Sigma_1(n, j) \cap \Sigma_n(n, j) = \langle a_{q(n)-j+1}, \dots, a_c \rangle
\]
is a central subgroup of $G$.

(3) Since $G$ decomposes as
\[
G = \langle a_1, \dots, a_{q(n)-j} \rangle \, Z(n, j) \, \langle a_{c+1}, \dots, a_n \rangle,
\]
a product of three independent abelian groups, given $g \in G$, we decompose $g$ uniquely as $g = g_1 g_2  g_3$ with  $g_i$ in the  $i$-th component above. Thus, $g \in Z(G)$  if and only if $g' = g_1  g_3 \in Z(G)$, equivalently, $g_1$  commutes with $\langle a_{c+1}, \dots, a_n \rangle$,  and $g_3$ with $\langle a_1, \dots, a_{q(n)-j} \rangle$.

Now, $g_1 = a_1^{i_1} \dotsm a_{q(n)-j}^{i_{q(n)-j}}$ commutes with $a_{c+1}$ if and only if $a_1^{i_1}$ does (and so, $i_1 = 0$). Similarly, we show $g_1$, $g_3$ are trivial.

We note that $\operatorname{rank}(Z(G)) = c - (q(n)-j) = 2q(n) + r(n) + j - (q(n)-j) = q(n) + r(n) + 2j \geq n/3$.

(4) Let $k^* = q(n) - j-1$.

Since
\[
[a_1, a_{c+1}]_{k} \in \langle a_{q(n)+1-j+k}, \dots, a_{q(n)+1+r(n)+2j+k} \rangle,
\]
and
\[
q(n)+1 + r(n) + 2j + k^* = q(n)+1 + r(n) + 2j + (q(n) - j -1 ) = c(n, j),
\]
we have
\[
[a_1, a_{c+1}]_{k} \in \langle a_{q(n)+1 - j + k}, \dots, a_{c(n, j)} \rangle,
\]
and therefore, $k^*$ is maximum $k$ such that
\[
[a_1, a_{c+1}]_{k} \in Z(G).
\]
We continue the proof by writing
\[
w_{c+1} = [a_1, a_{c+1}] = (a_l)^{u_l} \dots (a_{c})^{u_{c}},
\]
where \( u_i \in \mathbb{Z}_p \) and \( u_l \neq 0 \). The leading term of \( w_{c+1} \) is \( (a_l)^{u_l} \). Define
\[
\text{Base}\{w_{c+1}\} = \{a_l, \dots, a_{c}\}.
\]
Then,
\[
w_{c+1+k} = [a_1, a_{c+1}]_{k} = (a_{l+k})^{u_l} \dots (a_{c+k})^{u_c},
\]
for \( 0 \leq k \leq q(n) - j - 1 \). The leading term of \( w_{c+1+k} \) is \( (a_{l+k})^{u_l} \), and
\[
\text{Base}(w_{c+1+k}) = \{a_{l+k}, \dots, a_{c+k}\}.
\]
Thus,
\[
\bigcup \{\text{Base}(w_{c+1+k}) \mid 0 \leq k \leq q(n) - j - 1\}
\]
is contained in the commutative set \( \{a_{q+1-j}, \dots, a_{n-1}\} \). Since the leading terms of \( w_{c+1+k} \) for \( 0 \leq k \leq q(n) - j - 1 \) form the set
\[
\{(a_{l+k})^{u_l} \mid 0 \leq k \leq q(n) - j - 1\},
\]
and \( u_l \neq 0 \), the independence of
\[
\{[a_1, a_{c+1}] \mid 0 \leq k \leq q(n) - j - 1\}
\]
follows.
\end{proof}

\textbf{Example.} Let $G$ be the Heisenberg group \[G = \langle a_1, a_2, a_3 \mid (a_i)^m = 1, [a_1, a_3] = a_2, [a_1, a_2] =[a_2, a_3] = e\rangle,\] where $m$ is finite or infinite. Then $G$ is a $SSP$  of $G(3,2)$-type where $f:a_1\to a_2 \to a_3$. More generally, using Proposition \ref{prop3},  we  can construct the group $P_k$ ($k \geq 2$) with $3k$-generators, having a triple decomposition into $3$ abelian subgroups:
$$P_k = \langle b_1, \ldots, b_k \rangle \cdot \langle b_{k+1}, \ldots, b_{2k} \rangle \cdot \langle b_{2k+1}, \ldots, b_{3k} \rangle,$$
where the second factor is central, and where the Heisenberg relations hold: $[b_1, b_{2k+1}] = b_{k+1}, \ldots, [b_k, b_{3k}] = b_{2k}$ . Thus 
 $P_k = \langle b_1, b_{2k+1} \rangle \times \cdots \times \langle b_k, b_{3k} \rangle$ is isomorphic to $G^k$. If we consider $n=3k$, then $P_k$ is $SSP$ of $G(n,2k)$-type with canonical base $S = \{b_1, b_2, \ldots, b_k, b_{k+1}, \ldots, b_{2k}, b_{2k+1}, \ldots, b_{3k}\}$. Moreover,  the cut-off point of $P_k$ is $c(n,j)=2q(n)=2k$ ($j=0$) and an interesting observation is that  by Theorem \ref{cutpoint} $P_k$  does not lift to a $SSP$ group of type $G(n+1,c)$ since $c(n+1,j)=2q(n+1)+r(n+1)+j=2q(n)+1+j'$.

\begin{theorem}
Groups of SSP-type are metabelian.    
\end{theorem} 
\begin{proof}
We proceed by induction on the number of canonical polycyclic generators $n$ of an $SSP$-group.

From Section 2, the assertion is true for $1 \leq n \leq 6$. Suppose $G$ is a metabelian group of $G(n, c)$-type and $\widehat{G}$ is a lifting of $G$.

Then, $G = \langle a_1, \dots, a_c \rangle \langle a_{c+1}, \dots, a_n \rangle$, a product of two abelian subgroups.

Now, given the form of $c= 2q(n)+r(n)+j$,
it follows that $\langle a_{c+1},...,a_{n+1}\rangle$ is abelian  and therefore $\widehat{G}$ is a product of two abelian subgroups. It follows from Ito's theorem that $\widehat{G}$ is metabelian.
\end{proof}

\textbf{Observations.} Let $G$ be a group of $G(n, c)$-type. Then,
\begin{enumerate}
    \item[(1)] Since $\operatorname{rank}(Z(G)) \geq n/3$, $\operatorname{rank}(Z(G)) = 1$ if and only if $n = 3$. In particular, if $p$ is a prime number, and $G = C_p \wr C_p$, then $|Z_{i+1}(G) / Z_i(G)| = p$ for $i = 1, \dots, p-2$ and $|G / Z_{p-1}(G)| = p^2$. Thus, $G$ is an $SSP$-group if and only if $p = 2$, thus $G$ is a dihedral group of order $8$. Also $G / Z_i(G)$ is not an $SSP$-group, for $m$ odd and $i = 0, \dots, p-3$.

    \item[(2)] Suppose $G$ is a $G(n, c)$-group which lifts to $\widehat{G}$ a $G(n+1, c)$-group. Since
    \[
    G = \langle a_1, \dots, a_{q(n)-j} \rangle \, Z(G) \, \langle a_{c+1}, \dots, a_n \rangle
    \]
    and
    \[
    Z(G) = \langle a_{q(n)-j+1}, \dots, a_{c(n, j)} \rangle,
    \]
    in calculating the Hall-Witt relations $W(1, l, n+1)$ in the lifting $\widehat{G}$, we may restrict $a_l$ to $Z(G) \setminus Z(\widehat{G})$ (that is, $l = q(n)-j+1$), or $a_l$ to $G \setminus Z(G)$ (that is, $l$ in $\{1, \dots, q(n)-j\} \cup \{c+1, \dots, n\}$).

    If $l = q(n)-j+1$, then
    \[
    [a_{n+1}, a_1, a_{q(n)-j+1}][a_{q(n)-j+1}, a_{n+1}, a_1] = e,
    \]
    and as $[a_{n+1}, a_1] \in \langle a_2, \dots, a_n \rangle$,
    \[
    [a_{q(n)-j+1}, a_{n+1}, a_1] = e.
    \]

    If $1 \leq l \leq q(n)-j$, then $[a_1, a_l] = e$ and so
    \[
    [a_{n+1}, a_1, a_l][a_l, a_{n+1}, a_1] = e.
    \]

    If $c+1 \leq l \leq n$, then $[a_l, a_n] = e$, and so
    \[
    [a_1, a_l, a_{n+1}] [a_{n+1}, a_1, a_l] = e.
    \]

    Altogether: the $W(1, l, n+1)$ relations restrict to one of three kinds:
    \[
    [a_{q(n)-j+1}, a_{n+1}, a_1] = e;
    \]
    \[
     \quad [a_{n+1}, a_1, a_l] \cdot [a_l, a_{n+1}, a_1] = e, \quad \text{for } 1 \leq l \leq q(n)-j;
    \]
    \[
     \quad [a_{n+1}, a_1, a_l] \cdot [a_1, a_l, a_{n+1}] = e, \quad \text{for } c+1 \leq l \leq n.
    \]

\end{enumerate}

\section{$SSP$-groups of nilpotency class 2.}

 Let $R(n,j)$ denote the set of statements in the group $G$ of $G(n,c)$-type: 

\begin{itemize}
    \item[i)] $a_1^p=e$.
    \item [ii)] $[a_1,a_i]=e$, for any $2\leq i\leq c(n,j)$.
    \item[iii)] $[a_1, a_{c(n,j)+1+k}]   \in      \langle a_{q(n)+1-j},\ldots,a_{q(n)+1+r(n)+2j+k}\rangle$ for any $0\leq k \leq q(n) -j -1$ and, $[a_1,a_{c(n,j)+1}]$ is nontrivial.
\end{itemize}

\begin{theorem}\label{nilclass2}
Let $G$ be a non-abelian group of $G(n,c)$-type.
Then $G$ is of nilpotency class 2 if and only if
\begin{enumerate}
    \item $c = c(n,j)$ for some $0\leq j\leq q(n)-1$;
    \item $G$ has the presentation
    \[
    G = \langle a_1,...,a_{n} \ | \ R(n,j)^{\langle f\rangle}  \rangle.
    \]
\end{enumerate}
\end{theorem}

\begin{proof}
It is straightforward to see that a group with such a presentation is nilpotent of class 2. In the other direction, suppose $G$ has nilpotency class 2. The assertion is true for $n=3$. Suppose it is true for $n$.

As $G$ is metabelian, we have
\[
c = c(n,j) = 2q(n) + r(n) + j,\ \text{and}\ Z(G) = \langle a_{q(n)+1-j},\ldots,a_{c(n,j)} \rangle.
\]

Since $\widehat{G}$ is of $G(n+1,c)$-type, its essential relations are those $R(n,j)$ acquired from some group $G$ of $G(n, c)$-type,
\begin{align*}
(ii) \ \ & [a_1, a_i] = e\quad (2\leq i\leq c(n,j)), \\
(iii) \ \  & [a_1, a_{c(n,j)+1+k}] = a_{q(n)+1-j}^{x_k}\ldots a_{q(n)+1+r(n)+2j+k}^{y_k},\quad (0 \leq k \leq q(n)-j-1),
\end{align*}
together with
\begin{align*}
(iv) \ \  & [a_1, a_{n+1}] = a_2^{z_2}\cdots a_{q(n)+1-j}^{z_{q(n)+1-j}}  \cdots  a_{n-q(n)+j}^{z_{n-q(n)+j}}\cdots a_{n-1}^{z_{n-1}} 
\end{align*}
for some integers $x_k$, $y_k$, $z_l$.

Relations $(iii)$ and $(iv)$ reduce in $\widehat{G}$, in three stages to:
\begin{enumerate}
    \item $x_k = y_k= 0\quad (0\leq k \leq q(n) -j - 1)$;
    \item $z_2 = \ldots = z_{q(n)+1-j}=0$;
    \item $z_{n-q(n)+j} = \ldots = z_{n-1}=0$.
\end{enumerate}

These reductions are obtained as follows:
\[
e = [a_1, a_{c+1+k}, a_{n+1}] = [ a_{q(n)+1-j}^{x_k}, a_{n+1}] = ([a_1, a_{c+1}]^{x_k})_{(q(n)-j)}
\]
produces $x_k = 0$;

\[
e = [a_{n+1-c-k}, a_{n+1}, a_1] = [ a_{c+1}^{y_k}, a_1] = [a_1, a_{c+1}]^{-y_k}
\]
produces $y_k = 0$;

\[
e = [a_1, a_{n+1}, a_{c+1}]\ \text{produces}\ z_2 = 0,
\]
\[
\vdots
\]
\[
e = [a_1, a_{n+1}, a_{n+1}]\ \text{produces}\ z_{q(n)+1-j} = 0;
\]
\[
e = [a_1, a_{n+1}, a_{q(n)}]\ \text{produces}\ z_{n-1} = 0,
\]
\[
\vdots
\]
\[
e = [a_1, a_{n+1}, a_1]\ \text{produces}\ z_{n-q(n)+j} = 0.
\]
\end{proof}

\subsection{Matrix representations of class 2 nilpotent groups of $G(n,c)$-type}

We produce in this subsection faithful linear representations over $\mathbb{Z}_p$ of nilpotent groups of $G(n,c)$-type. This goes to show that the presentations given for groups in Theorem \ref{nilclass2} are consistent, and therefore all have Hirsch length $n$.

As we have seen, if $G$ is a classe 2 nilpotent group of $G(n,c)$-type, $G$ has a factorization by abelian subgroups $G = V\rtimes K$, where $V$ is a normal subgroup of $G$, and $V$ itself factors as $V = W\oplus
Z$. The abelian subgroups are : W = $\langle a_1,...,a_{q(n) - j}\rangle$; Z = $\langle a_{q(n)+1 - j},...,a_{c(n,j)}\rangle$, the center of $G$; $K = \langle a_{c(n,j)+1},...,a_n\rangle$.

Following the characterization of $G$, we consider $V$ as an $\mathbb{Z}_p [K]$ module written additively and construct matrix  $[a_s]$ representing elements of $K$. Let $u_k$ be the additive form of the right-hand side of the equation  \[ [a_1, a_{c(n,j)+1+k}] = (a_{q(n)+1-j})^{x_k}....(a_{q(n)+1+r(n)+2j+k})^{y_k},\]
where $0 \leq k \leq q(n)-j-1$. The action of $K$ on $V$ is as follows:
\begin{align*}
[a_{c(n,j)+1}] : a_1 \mapsto a_1 + u_0 , a_i \mapsto a_i \quad (2 \leq i \leq c(n,j))
\end{align*}
\[
[a_{c(n,j)+2}] : a_1 \mapsto a_1 + u_1 , a_2 \mapsto a_2 + (u_0)^f, a_i \mapsto a_i \quad (3 \leq i \leq c(n,j))
\]
where we recall that $u_0$  is non-trivial. In general,
\begin{align*}
[a_{c(n,j)+k}] : a_1 &\mapsto a_1 + u_{k-1} , \ldots , a_k \mapsto a_k + (u_0)^{f^{k-1}} , a_i \mapsto a_i \quad (k+1 \leq i \leq c(n,j))
\end{align*}

Thus, the matrix representation in block form of $[a_{c(n,j)+k}]$ is
\begin{align*}
\begin{bmatrix}
I_{(q(n) -j) \times (q(n) -j)} & [B_k] \\
0_{(q(n)+r(n)+2j) \times (q(n)-j)} & I_{(q(n)+r(n)+2j) \times (q(n)+r(n))}
\end{bmatrix}
\end{align*}
where
\begin{align*}
[B_k] =
\begin{bmatrix}
u_k & 0 & 0 & \ldots & 0 \\
0 & u_{k-1} & 0 & \ldots & 0 \\
\vdots & 0 & \ddots & \vdots & \vdots \\
\vdots & \vdots& \ldots & u_0 & 0 \\
0 & 0 & \ldots & 0 & 0
\end{bmatrix}
\end{align*}

The groups $G$ constructed above are of affine sort and can be represented matricially as
\begin{align*}
\begin{bmatrix}
1 & V \\
0_{1 \times (q-j)} & [K]
\end{bmatrix}
\end{align*}
where $V$ is in row form.

Finally, in the affine representation of $G$, the subgroup which corresponds to $H$ is $\langle a_1,..., a_{c(n,j)}\rangle\langle [a_{c(n,j)+1}],...,[a_{n-1}]\rangle$, that to $H^f$ is $$\langle a_2,..., a_{c(n,j)+1}\rangle\langle [a_{c(n,j)+2}],...,[a_n]\rangle,$$ and what corresponds to the monomorphism $f$ is the extension of the obvious isomorphisms
\begin{align*}
f_0 : \langle a_1,..., a_{c(n,j)}\rangle &\to \langle a_2,..., a_{c(n,j)+1}\rangle, \
f_1 : \langle [a_{c(n,j)+1}],...,[a_{n-1}]\rangle &\to \langle [a_{c(n,j)+2}],...,[a_n]\rangle.
\end{align*}

\section{Lifting SSP groups having nilpotency class 2}

We begin this section with a solution of a specific $f$-equation over $\mathbb{Z}_p$.

Let $B = \langle b_1,\ldots,b_n \rangle$ be a multiplicative free abelian $p$-group ($p$ prime or infinite) of rank $n \geq 2$, which is self-similar by
\[
f: b_i \rightarrow b_{i+1} \quad (1 \leq i \leq n-1).
\]

Let
\[
w = b_1^{x_1} \ldots b_s^{x_s}
\]
be a non-trivial element of $B$ such that $1 \leq s \leq n-1$.

Given $1 \leq t \leq n - s$, and $u,v$ non-negative in $\mathbb{Z}_p$ (if $p$ is a prime number, we view $\mathbb{Z}_p$ as $\{0,\ldots,p-1\}$), let
\[
w^{u + v f^t} = \left( b_1^{x_1 u} \ldots b_s^{x_s u} \right) \left( b_{1+t}^{x_1 v} \ldots b_{s+t}^{x_s v} \right).
\]

\begin{prop}
Suppose $w$ is nontrivial, and $w^{u + v f^t} = e$. Then, $u = v = 0$.    
\end{prop}
\begin{proof}
    $w = b_1^{x_1} \ldots b_s^{x_s}$,
\[
w^{f^t} = b_{1+t}^{x_1} \ldots b_{s+t}^{x_s}.
\]

(1) The assertion is true as the set of indices $\{1,\ldots,s\}$, $\{1+t,\ldots,s+t\}$ of the pair $w$, $w^{f^t}$ are disjoint (that is, $s < 1 + t$).

(2) Suppose $1 + t \leq s$. Then, from
\[
w^{u + v f^t} = \left( b_1^{x_1} \ldots b_t^{x_t} \right)^u \left( b_{1+t}^{x_{1+t} u + x_1 v} \ldots b_s^{x_s u + x_{s - t} v} \right) \left( b_{1+s}^{x_{1+s - t}} \ldots b_{s+t}^{x_s} \right)^v
\]
we conclude
\begin{itemize}
    \item[(i)] $\{ x_1,\ldots,x_t \} u = \{ 0 \}$,
    \item[(ii)] $L = \{ x_{t+1} u + x_1 v, \ldots, x_s u + x_{s - t} v \} = \{ 0 \}$,
    \item[(iii)] $\{ x_{1+s - t}, \ldots, x_s \} v = \{ 0 \}$.
\end{itemize}

Write $s$ modulo $t$ as $s = s' t + r'$ with $1 \leq s'$, $0 \leq r' \leq t-1$, and accordingly, partition $L$ into $s'$ $t$-blocks followed by an $r'$-block. Thus,
\[
L = \bigcup_{i=1}^{s'} L_i \cup M,
\]
where
\begin{align*}
L_1 &= \{ x_{t+1} u + x_1 v, \ldots, x_{2t} u + x_t v \}, \\
L_2 &= \{ x_{2t+1} u + x_{t+1} v, \ldots, x_{3t} u + x_{2t} v \}, \\
&\vdots \\
L_i &= \{ x_{i t + 1} u + x_{(i - 1) t + 1} v, \ldots, x_{(i + 1) t} u + x_{i t} v \}, \\
M &= \{ x_{s' t + 1} u + x_{(s' - 1) t + 1} v, \ldots, x_s u + x_{s - t} v \}.
\end{align*}

Suppose $u$ is not zero.

Then, from (i)
\[
x_1 = \ldots = x_t = 0.
\]
Therefore,
\[
L_1 = \{ x_{t+1} u + x_1 v, \ldots, x_{2t} u + x_t v \} = \{ x_{t+1} u, \ldots, x_{2t} u \} = \{ 0 \};
\]
thus,
\[
x_{t+1} = \ldots = x_{2t} = 0.
\]

By induction,
\[
L_{s'-1} = \{ x_{(s'-1) t + 1} u + x_{(s'-2) t + 1} v, \ldots, x_{s' t} u + x_{(s'-1) t} v \} = \{ x_{(s'-1) t + 1} u, \ldots, x_{s' t} u \} = \{ 0 \}.
\]

Further,
\[
M = \{ x_{s' t + 1} u + x_{(s' - 1) t + 1} v, \ldots, x_s u + x_{s - t} v \} = \{ x_{s' t + 1} u, \ldots, x_s u \} = \{ 0 \}.
\]

Hence,
\[
w = e;
\]
a contradiction. If $u = 0$, then we directly obtain $v = 0$.
\end{proof}

\begin{theorem}\label{SSPHnil}
Let $G$ be a group of type $G(n, c(n,j))$ of nilpotency class 2, having essential relations $R(n,j)$, as in Theorem \ref{nilclass2}.  Suppose $\widehat{G}$ is a lifting of $G$. Then
$\widehat{G} = \langle a_1,...,a_{n+1} \ | \   \widehat{R}^{\langle F\rangle} \rangle$ ,
where \[F: a_1\to a_2\to \ldots \to a_n\to a_{n+1},\]
and $\widehat{R}$ is the set of essential relations :
\begin{itemize}
    \item[(i)] $(a_1)^p =e;$
    \item[(ii)] $[a_1, a_i] =e \ \ (2\leq i\leq  c(n,j));$
    \item[(iii)] $[a_1, a_{c(n,j)+1+k}] \in 
\langle a_{q(n)+2-j},\ldots,a_{q(n)+r(n) +2j +k}\rangle \ \ 
0\leq k \leq q(n) -j -1\rangle$, and is nontrivial if $k=0$;
\item[(iv)] $[a_1, a_{n+1}] \in 
\langle a_{q(n)+1-j},\ldots,a_{c(n,j)+1}\rangle.$
\end{itemize}    
\end{theorem}

We note that $(iii)$ implies that for $1\leq t<s \leq n$ we have \[[a_t,a_s]=[a_1,a_{s-t+1}]_{t-1} \in 
Z(G) = \langle a_{q(n)-j+2}, \dots, a_{c(n,j)} \rangle.
\] 
Furthermore, we note
\[
[a_1, a_{n+1}, a_1] \in \langle [a_{c(n,j)+1}, a_1] \rangle = \langle a_{q(n)+2-j}, \dots, a_{q(n)+r(n)+2j} \rangle,
\] and 
\[
[a_1, a_{n+1}, a_{n+1}] \in \langle [a_{q(n)+1-j}, a_{n+1}] \rangle = \]\[=\langle [a_1, a_{c(n,j)+1}]_{q(n)-j} \rangle
= \langle a_{q(n)+2-j}, \dots, a_{q(n)+r(n)+2j} \rangle_{q(n)-j} =\] \[=\langle a_{2q(n)+2-2j}, \dots, a_{2q(n)+r(n)+j} \rangle
= \langle a_{2q(n)+2-2j}, \dots, a_{c(n,j)} \rangle,
\]
and the right-hand side of both are subgroups of $Z(G)$. Thus, once the theorem is established it would follow that
\[
\gamma_3(G) = \langle [a_1, a_{n+1}, a_1], [a_1, a_{n+1}, a_{n+1}] \rangle
\]
and $G$ is nilpotent of class at most 3.

\subsection{Proof of Theorem 6.2}
\begin{proof}
We use the presentation of $G$ given in Theorem \ref{nilclass2},
\[
G = \langle a_1, a_2, \dots, a_n \mid  R(n,j)^{\langle f \rangle}  \rangle
\]
where $R(n,j)$ are the essential relations:
\begin{enumerate}
    \item[(i)]  $a_1^p=e$
    \item[(ii)] $[a_1, a_i] = e \quad (2 \leq i \leq c(n,j))$,
    \item[(iii)] $[a_1, a_{c(n,j)+1+k}] = a_{q(n)+1-j}^{x_k} a_{q(n)+2-j}^{x'k} \dots a_{q(n)+r(n)+2j+k}^{y'k} a_{q(n)+r(n)+2j+k}^{y_k}, \\ \quad (1 \leq k \leq q(n)-j-1)$, and is nontrivial if $k=0$.
\end{enumerate}
Then, as before, we have a presentation of $\widehat{G}$ by adjoining $a_{n+1}$ to $\{a_1, a_2, \dots, a_n\}$ and the relation
\[
(iv): [a_1, a_{n+1}] = (a_2)^{z_2} (a_3)^{z_3} \dots (a_{n-2})^{z_{n-2}} (a_{n-1})^{z_{n-1}}
\]
to the essential relations $R(n,j)$, for some integers $x_k$, $y_k$, $z_l$, and take the closure of these relations under powers of $f$.

Recall,
\[
Z(G) = \langle a_{q(n)-j+1}, \dots, a_{c(n,j)} \rangle
\]
\[
Z(\widehat{G}) = \langle a_{q(n)-j+2}, \dots, a_{c(n,j)} \rangle.
\]
The theorem will follow from two sorts of reductions:
\begin{enumerate}
    \item $x_0 = y_0 = 0$;
    \item
        \begin{itemize}
            \item $x_k = z_{k+1} = 0$,
            \item $y_k = z_{n-k} = 0 \quad (1 \leq k \leq q(n)-j-1)$.
        \end{itemize}
\end{enumerate}

The first was already done in Theorem \ref{cutpoint}, for a general nilpotency class. We will prove (2) by induction on $k$. Let $k=1$. Then,
\[
[a_1, a_{c(n,j)+2}] = (a_{q(n)+1-j})^{x_1} \dots (a_{q(n)+r(n)+2j})^{y_1},
\]
\[
[a_1, a_{c(n,j)+2}, a_{n+1}] = [(a_{q(n)+1-j})^{x_1}, a_{n+1}]
= ([(a_1, a_{c(n,j)+1}]^{x_1})_{q(n)-j},
\]
\[
[a_{n+1}, a_1, a_{c(n,j)+2}] = [((a_2)^{z_2} \dots (a_{q(n)})^{z_{q(n)}} (a_{c(n,j)+1})^{z_{c(n,j)+1}} \dots (a_n)^{z_n})^{-1}, a_{c(n,j)+2}]
=\] \[ [((a_2)^{-z_2}, a_{c(n,j)+2}] = ([(a_1), a_{c(n,j)+1}]^{-z_2})_1.
\]
Therefore, relation $W'(1, c(n,j)+2, n+1)$ reduces to
\[
([(a_1, a_{c(n,j)+1}]^{x_1})_{q(n)-j} \cdot ([a_1, a_{c(n,j)+1}]^{-z_2})_1 = e.
\]
Write $w=[a_1, a_{c(n,j)+1}]$ and $t=q(n)-j-1$. Then,
\[
w^{(x_1 \cdot f^t) - z_2} = e.
\]
It follows from the previous proposition that
\[
x_1 = 0, z_2 = 0.
\]
In a similar fashion, the symmetric $W'(1, q-j-2, n+1)$, yields
\[
y_1 = 0, z_{n-1} = 0.
\]
Thus, after steps $k=0,1$, the system $(iii)$ is reduced to

\begin{itemize}
    \item $[a_1, a_{c(n,j)+1}] = (a_{q(n)+2-j})^{x'_0} \dots (a_{q(n)+r(n)+2j+1})^{y'_0}$ nontrivial,
    \item $[a_1, a_{c(n,j)+1+1}] = (a_{q(n)+2-j})^{x'_k} \dots (a_{q(n)+r(n)+2j+k-1})^{y'_k}$,
    \item $[a_1, a_{c(n,j)+1+k}] = (a_{q(n)+1-j})^{x_k} (a_{q(n)+2-j})^{x'k} \dots (a_{q(n)+r(n)+2j+k})^{y'k} (a_{q(n)+r(n)+2j+k})^{y_k} \\ (2 \leq k \leq q(n)-j)$.
\end{itemize}
and $(iv)$ reduces to $[a_1, a_{n+1}] = (a_3)^{z_3} \dots (a_{n-2})^{z_{n-2}}.$ The argument proceeds similarly for a general $k$.
\end{proof}
\begin{corollary}
Let $G$ be an SSP-group where $o(a_i) = 2$. Then $G$ is a finite 2-group of nilpotency class at most 2.
\end{corollary}
\begin{proof}
We proceed by induction on $n$. The assertion is true for $n = 2$. Suppose the assertion holds for $n$ and let $G$ be of type $G(n+1, c(n,j))$. Then $H$ has nilpotency class at most 2.

Following the proof of the previous theorem, we note that the right-hand side of $(iii)$ is an element of the elementary abelian 2-group $\langle a_{q(n)-j+1}, \dots, a_{c(n,j)+1} \rangle$.

As $a_1^2 = a_{n+1}^2 = [a_1, a_{n+1}]^2 = e$, it follows that $[a_1, a_{n+1}, a_1] = [a_1, a_{n+1}, a_{n+1}] = e$. By Theoreom \ref{SSPHnil} 
\[
\gamma_3(G) = \langle [a_1, a_{n+1}, a_1], [a_1, a_{n+1}, a_{n+1}] \rangle.
\]Then $\gamma_3(G)=1$ and the proof is complete.
\end{proof}

\end{document}


\title{Appendix: Computational experiments}
\author{Bettina Eick}
\date{}
\maketitle

Throughout, let $p$ be a prime and write $\Z_p = \{0, \ldots, p-1\}$.

Let $P$ be a presentation on $n$ generators $g_1, \ldots, g_n$ 
with relations of the form
\begin{eqnarray*}
g_i^p &=& 1 \;\; \mbox{ for } 1 \leq i \leq n, \\
{[g_i, g_j]} &=& g_{i+1}^{t_{i,j,i+1}} \cdots g_{j-1}^{t_{i,j,j-1}} \;\;
                 \mbox{ for } 1 \leq i < j \leq n,
\end{eqnarray*}
for certain $t_{i,j,k} \in \Z_p$. Then $P$ defines a 
group of order dividing $p^n$. We say that $P$ is a {\em SSP presentation}
if $t_{i,j,k} = t_{i-1,j-1,k-1}$ holds for all $i,j,k$ with $i < 1$. 
We say that $P$ is {\em consistent} if it defines a group of order 
precisely $p^n$. 

\begin{lemma}
If $G$ is a SSP-group of order $p^n$, then there exists a SSP presentation
$P$ on $n$ generators defining $G$.
\end{lemma}

\begin{proof}
By Theorem 1.4 the group $G$ has a polycyclic generating set $S = \{g_1, 
\ldots, g_n\}$ such that its elements have order $p$ and the generating 
set is self-similar under $f : g_i \ra g_{i+1}$ for $1 \leq i \leq n-1$. 
As $S$ is a polycyclic generating set, it corresponds to a polycyclic
presentation with relations $g_i^p = 1$ for $1 \leq i \leq n$ and 
commutator relations expressing $[g_i, g_j]$ as a word in the generators
for $1 \leq i < j \leq n$, see Chapter 8 in \cite{HEO05}. Using the
self-similarity, a commutator relation $[g_i, g_j]$ corresponds to the
relation for $[g_{i-1}, g_{j-1}]$ and thus $t_{i,j,k} = t_{i-1,j-1,k-1}$
for $i > 1$ follows. Using Theorem 2.3 it follows that the exponents
$t_{i,j,k}$ are zero if $k \leq i$ or $k \geq j$. This yields the desired
result.
\end{proof}

We describe an algorithm to determine all consistent SSP presentations on
$n$ generators. This uses induction on $n$. The initial
step of the induction is given by the SSP presentations on $2$ generators;
note that there is exactly one such presentation and this corresponds to the
elementary abelian group of order $p^2$. The following elementary lemma 
underpins the induction step.

\begin{lemma}
If $P$ is a consistent SSP presentation on $n$ generators $g_1, \ldots, 
g_n$, then the induced presentation $\ol{P}$ on the generators 
$g_1, \ldots, g_{n-1}$ is a consistent SSP presentation. 
\end{lemma}

Thus the induction step considers all consistent SSP presentations on
$n-1$ generators, and extends each one of them in turn in all possible 
ways to an SSP presentation on $n$ generators. Given an SSP presentation 
$\ol{P}$ on $n-1$ generators $g_1, \ldots, g_{n-1}$ and an additional 
generator $g_n$, the possible extensions of $\ol{P}$ to a SSP presentation 
$P$ on the $n$ generators $g_1, \ldots, g_n$ only depend on the single 
essential relation
\[ [g_1, g_n] = g_2^{a_2} \cdots g_{n-1}^{a_{n-1}},\]
with $a = (a_2, \ldots, a_{n-1}) \in \Z_p^{n-2}$. We investigate which
vectors $a$ lead to consistent SSP presentations $P$ on $n$ generators.

\begin{theorem}
\label{extend}
Let $\ol{P}$ be a consistent SSP presentation on $n-1$ generators $g_1,
\ldots, g_{n-1}$ defining the group $G$. Let $P$ be
the SSP presentation arising from $\ol{P}$ by adding the generator $g_n$
and the relation $[g_1, g_n] = g_2^{a_2} \cdots g_{n-1}^{a_{n-1}}$ for
some $a \in \Z_p^{n-2}$. For $1 \leq i \leq n-1$ define 
\[ w_i = g_{i+1}^{t_{i,n,i+1}} \cdots g_{n-1}^{t_{i,n,n-1}},\]
with $t_{1,n,k} = a_k$ and $t_{i,n,k} = t_{i-1,n-1,k-1}$ in $\ol{P}$ 
for $i > 1$. Then $P$ is consistent if and only if
\[ \alpha : G \ra G : g_i \ms g_i \cdot w_i\]
induces an automorphism of $G$. In this case the presentation $P$ 
defines the group $G \rtimes_\alpha C_p$.
\end{theorem}

\begin{proof}
If $P$ is consistent, then it defines a group $H$ of order $p^n$ on the
generators $g_1, \ldots, g_n$. In this case $\alpha$ is the automorphism
of $G \leq H$ arising from conjugation with $g_n$. 
Conversely, if $\alpha$ is an automorphism of $G$, then $H = 
G \rtimes_\alpha C_p$ is well-defined. The group $H$ has a power-commutator
presentation on its $n$ generators $g_1, \ldots, g_n$ and this coincides
with the presentation $P$. Hence $P$ is consistent.
\end{proof}

We consider Theorem \ref{extend} in more detail. Given $a \in \Z_p^{n-2}$
and a consistent SSP presentation $\ol{P}$ on $n-1$ generators, it is 
easy to determine $\alpha$. Checking whether $\alpha$ extends 
to an automorphism corresponds to evaluating the defining relations of 
$G$ in the images $g_1 w_1, \ldots, g_{n-1} w_{n-1}$ which is 
straightforward.

In summary, we obtain an effective method to determine all consistent 
SSP presentations on $n$ generators. We summarize its application for 
$p \in \{3,5,7\}$ as follows. Note that all groups determined in this
application have class at most $3$ and the following tables list the 
numbers of presentations by the order and the class of the group they 
define.

\begin{figure}[htb]
\begin{center}
\begin{tabular}{|l|r|r|r|}
\hline
$n$ & class 1 & class 2 & class 3 \\
\hline
\hline
2   & 1 & 0 & 0 \\
3   & 1 & 2 &  0 \\
4   & 1 & 8 & 0 \\
5   & 1 & 26 &  0 \\
6   & 1 & 98 & 144 \\
7   & 1 & 458 &  1728 \\
8   & 1 & 2834 & 16848 \\
9   & 1 & 22112 & 159408 \\
10  & 1 & 200474 & 1551312 \\
\hline
\end{tabular}
\end{center}
\caption{Numbers of consistent SSP presentations
         defining groups of order $3^n$ and class 1, 2 or 3}
\end{figure}

\begin{figure}[htb]
\begin{center}
\begin{tabular}{|l|r|r|r|}
\hline
$n$ & class 1 & class 2 & class 3 \\
\hline
\hline
2   & 1 & 0 & 0 \\
3   & 1 & 4 &  0 \\
4   & 1 & 24 & 0 \\
5   & 1 & 124 &  0 \\
6   & 1 & 724 & 2400 \\
7   & 1 & 6124 & 72000 \\
8   & 1 & 93124 & 1860000 \\
\hline
\end{tabular}
\end{center}
\caption{Numbers of consistent SSP presentations 
         defining groups of order $5^n$ and class 1, 2 or 3}
\end{figure}

\begin{figure}[htb]
\begin{center}
\begin{tabular}{|l|r|r|r|}
\hline
$n$ & class 1 & class 2 & class 3 \\
\hline
\hline
2   & 1 & 0 & 0 \\
3   & 1 & 6 &  0 \\
4   & 1 & 48 & 0 \\
5   & 1 & 342 &  0 \\
6   & 1 & 2694 & 14112 \\
7   & 1 & 33270 & 790272 \\
\hline
\end{tabular}
\end{center}
\caption{Numbers of consistent SSP presentations
         defining groups of order $7^n$ and class 1, 2 or 3}
\end{figure}

\begin{remark}
Let $P$ be a SSP presentation on $n$ generators with respect to the
prime $p$. By eliminating the relations $g_i^p = 1$ one obtains a 
new SSP presentation $\ol{P}$. The nilpotent quotient algorithm 
\cite{NQ} allows to determine the Hirsch length of the infinite
nilpotent group defined by $\ol{P}$. By construction this Hirsch
length is at most $n$, and could potentially be smaller.

Applying this elimination idea to the consistent SSP presentations on 
at most $8$ generators for the prime $3$ yields that all of these 
presentations $P$ induce new presentations $\ol{P}$ defining torsion 
free nilpotent SSP groups of Hirsch length exactly $n$.
\end{remark}

\bibliographystyle{abbrv}

\bigskip
\bigskip

Bettina Eick \\
Institut für Analysis und Algebra \\
Technische Universität Braunschweig \\
Universitätsplatz 2 \\
38106 Braunschweig \\
Germany \\
email: beick@tu-bs.de